\documentclass[10pt]{amsart}

\usepackage{amsmath,graphics}
\usepackage{amsfonts,amssymb}
\usepackage{amsmath,amscd}
\usepackage[OT2,T1]{fontenc}
\DeclareSymbolFont{cyrletters}{OT2}{wncyr}{m}{n}
\DeclareMathSymbol{\Sha}{\mathalpha}{cyrletters}{"58}

\theoremstyle{plain}
\newtheorem*{theorem*}{Theorem}
\newtheorem*{lemma*} {Lemma}
\newtheorem*{corollary*} {Corollary}
\newtheorem*{proposition*} {Proposition}
\newtheorem{theorem}{Theorem}[section]
\newtheorem{lemma}[theorem]{Lemma}
\newtheorem{corollary}[theorem]{Corollary}
\newtheorem{proposition}[theorem]{Proposition}

\theoremstyle{remark}

\newtheorem*{definition}{Definition}

\theoremstyle{definition}

\begin{document}

\title{Deformations of Galois representations and the theorems of Sato-Tate, Lang-Trotter et al}
\author{Aftab Pande}
\date{\today}
\begin{abstract}

 We construct infinitely ramified Galois representations $\rho$ such that the $a_l (\rho)$'s have distributions in contrast to the statements of Sato-Tate, Lang-Trotter and others. Using similar methods we deform a residual Galois representation for number fields and obtain an infinitely ramified representation with very large image, generalising a result of Ramakrishna.
 \end{abstract}
\maketitle

\section{Introduction}

If $E$ is an elliptic curve over $\mathbb{Q}$, then by the Mordell-Weil theorm we know that the set of rational points $E(\mathbb{Q})$ is a finitely generated abelian group. For $l$ a prime, let $\mathbb{F}_l$ be the prime field and define $a_l (E) := l + 1 - \#E(\mathbb{F}_l)$. The integers $a_l (E)$ provide us a lot of information about the connections that exist between elliptic curves, modular forms and Galois representations. 

If $f$ is a classical modular form then it has a fourier expansion $f( \tau) = \sum_n a_n (f) q^n$, where $\tau \in \mathfrak{h}$ and $q = e^{2 \pi i \tau}$. The celebrated Shimura-Taniyama conjecture which said that all elliptic curves are modular could be rephrased as $a_l (E) = a_l (f)$. To each elliptic curve $E$ we can also associate a Galois representation $\rho_E : G \rightarrow GL_2 (\mathbb{Z}_p)$, ramified at finitely  many primes, where $G = Gal(\overline{\mathbb{Q}} / \mathbb{Q}) $. If we define $a_l (\rho_E) := Tr (\rho_E (\sigma_l))$, where $\sigma_l$ is the Frobenius at $l$, then we know that $a_l (\rho_E) = a_l (E)$. Serre conjectured \cite{S3} (now a theorem due to Khare and Wintenberger) that every odd, irreducible Galois representation $\rho$ is modular, which could be rephrased as $a_l (\rho) = a_l (f)$ for a certain modular form $f$ of weight $2$. 

The $a_l (E)$ (for simplicity we denote them $a_l$) have also been studied asymptotically as $l$ varies. Hasse showed that $| a_l| \leq 2 \sqrt{l}$ and we are interested in the following statements. 
 
\begin{itemize}
 \item (Sato-Tate) If $\alpha_l = \frac{a_l}{2\sqrt{l}}$, then the probability distribution function for the $\alpha_l$ is given by  $P(l | A \leq \alpha_l \leq B) = \frac{2}{\pi} \int_A^B \sqrt{1-x^2}dx$.

 \item (Lang-Trotter) $\# \{ l < x | a_l = D \} = O(\frac{\sqrt{x}}{log(x)})$, where $D \in \mathbb{Z}$ is a fixed integer.

\item (Serre and Elkies) $\#\{ l < x | a_l = 0 \} < O(x^{3/4})$.

\item (Serre) $\#\{ l < x | |a_l| = 1 \} < O(\frac{x}{(log x)^{1 + \epsilon}})$, for any $\epsilon < \frac{1}{3}$.

\end{itemize}

The Sato-Tate \cite{T1} conjecture, assuming multiplicative reduction at some prime, was proved recently by Taylor, Clozel, Harris, Shepherd-Barron (\cite{T2}, \cite{HST}) and for an elementary introduction about the significance of the conjecture read \cite{M3}. There is also a recent preprint (\cite{BGHT}) by Taylor, Harris, Geraghty and Barnet-Lamb which seems to have proved a generalized version of the Sato-Tate conjecture. In the case that $a_l = 0$ (also known as supersingular primes), the Lang-Trotter \cite{LT} conjecture assumes that $E$ has no complex multiplication. Serre \cite{S1} showed supersingular primes have density zero using the Cebotarev Density theorem and then proved \cite{S2} the above estimate assuming the generalised Riemann Hypothesis. Elkies \cite{E1} showed that there are infinitely many supersingular primes and later  proved \cite{E2} Serre's  estimate without assuming the GRH. When $a_l = 1$, Mazur refers to the corresponding $l$'s as anomalous primes, and the estimate was established by Serre \cite{S2}.

Our goal is to construct representations $\rho : G \rightarrow GL_2 (\mathbb{Z}_p)$, ramified at infinitely many primes, such that the $a_l  (\rho)$'s have distributions in contrast with the above 4 statements. We obtain the following results:

\begin{itemize}
\item For any $\epsilon \leq 1$, there exists a deformation $\rho : G \rightarrow GL_2 (\mathbb{Z}_p)$ such that the set $R = \{ l | |\frac{a_l(\rho)}{2\sqrt{l}}| \leq \epsilon \}$ is of density one.

\item If $D \in \mathbb{Z}$ is fixed, we can find a deformation $\rho : G \rightarrow GL_2 (\mathbb{Z}_p)$, such that $\# \{ l < x | a_l(\rho) = D \} \ne O(\frac{\sqrt{x}}{log x})$.

\item There exists a deformation $\rho : G \rightarrow GL_2 (\mathbb{Z}_p)$, such that $\# \{l <x | a_l(\rho)= 0 \} \ne O(x^{3/4})$.  

\item There exists a deformation $\rho : G \rightarrow GL_2 (\mathbb{Z}_p)$, such that $\# \{l <x | a_l(\rho) = 1 \} \ne (\frac{x}{(log x)^{1+\epsilon}})$, where $\epsilon < 1/3$.

\end{itemize}

Serre had shown examples of representations ramified at all primes, but in his case the representations were reducible. Ramakrishna \cite{R2} showed the existence of irreducible representations, ramified at infinitely many primes. Infinitely ramified representations do not arise from geometry, but as the limit of geometric representations. They have been useful in showing the independence of the conditions of potentially semistable and finite ramification (as in the Fontaine-Mazur conjecture), and also in constructing representations which are pure or rational as in \cite{KLR}. 

To construct such representations, we make use of techniques which we develop in section 3. In developing these techniques we generalise a result of Ramakrishna \cite{R4} and obtain the following result.

\begin{theorem}
Let $\mathbf{k}$ be a finite field of characteristic $p$, $W(\mathbf{k})$ the ring of Witt vectors of $\mathbf{k}$, $F$ a number field such that $[F(\mu_p):F] = p-1$ and $G_F$ its absolute Galois group Gal$ (\overline{F}/F)$. If $\overline{\rho} : G_F \rightarrow GL_2(\mathbf{k})$ is a representation whose image contains $SL_2(\mathbf{k})$ then there exists a deformation of $\overline{\rho}$, ramified at infinitely many primes, $\rho : G_F \rightarrow GL_2 (W(\mathbf{k})[[T_1, T_2,.., T_r,....,]])$, whose image is full i.e. contains $SL_2 (W(\mathbf{k})[[T_1, T_2,.., T_r, ....,]])$.
\end{theorem}

\section{Deformation Theory}

We briefly recall some facts of deformation theory, and refer the reader to Mazur \cite{M1} and Ramakrishna \cite{R1}, \cite{R3}, \cite{R4} for more details.

Let $\overline{\rho}  : G \rightarrow GL_2 (\mathbf{k})$ be an absolutely irreducible representation of a profinite group $G$, and let $R$ be a ring in the category $\mathfrak{C}$ of Artinian local rings with residue field $\mathbf{k}$. If $\gamma$ is a lift of $\overline{\rho}$ to $GL_2 (R)$, then two lifts $\gamma_1, \gamma_2$ are equivalent if $\gamma_1 = A \gamma_2 A^{-1}$, where $A$ is congruent to the identity matrix modulo $m_R$, the unique maximal ideal of $R$. A deformation $\gamma$ of $\overline{\rho}$ to $R$ is an equivalence of lifts of $\overline{\rho}$ to $R$. We have the following theorem due to Mazur \cite{M2}:

\begin{theorem}

There is a complete local Noetherian ring $R^{un}$ with residue field $\mathbf{k}$ and a continuous deformation $\widetilde{\rho} : G \rightarrow GL_2 (R^{un})$ such that:

\begin{itemize}

\item Reduction of $\widetilde{\rho}$ modulo the maximal ideal of $R^{un}$ gives $\overline{\rho}$

\item For any ring $R$ in $\mathfrak{C}$ and any deformation $\gamma$ of $\overline{\rho}$ to $GL_2(R)$ there is a unique homomorphism $\phi : R^{un} \rightarrow R$ in $\mathfrak{C}$ such that $\phi \circ  \widetilde{\rho} = \gamma$ as deformations.

\end{itemize}

\end{theorem}

$R^{un}$ is called the universal deformation ring if $\overline{\rho}$ is irreducible (or if the centraliser of the image of $\overline{\rho}$ is exactly the scalars), and the versal ring otherwise. In this paper we fix the determinants of all deformations of $\overline{\rho}$, which means we study the cohomology of $Ad^0 \overline{\rho}$ (the $2 \times 2$ matrices of trace zero, with $G$ acting via conjugation by $\overline{\rho}$) and not $Ad \overline{\rho}$. It is known that $R^{un}$ is a quotient of $W(\mathbf{k}) [[ T_1, T_2,..., T_r]]$, where $r = dim_{\mathbf{k}}H^1(G, Ad^0 \overline{\rho})$. For the rest of the paper, we let $X$ denote $Ad^0 \overline{\rho}$.

Let $G_F$ be the Galois group of a number field $F$, $\rho_n : G_F \rightarrow GL_2 (W (\mathbf{k})/p^n)$ be a deformation of $\overline{\rho}$, and suppose we want to deform $\rho_n$ to $\rho_{n+1} : G_F \rightarrow GL_2 (W (\mathbf{k})/p^{n+1})$. When we fix determinants, the obstructions to deformation lie in $H^2(G_F, X)$. For a given $\overline{\rho}$ and a finite set of primes $S$ containing $p$ and the ramified primes of $\overline{\rho}$, define $\Sha^i_S (X)$ to be the kernel of the map $H^i(G_{F,S}, X) \rightarrow \oplus_{\mathfrak{q} \in S} H^i(G_{\mathfrak{q}}, X)$, where $G_{F,S}$ is the Galois group of the maximal extension of $F$ unramified outside $S$, and $G_{\mathfrak{q}} = Gal (\overline{F}_{\mathfrak{q}}/F_{\mathfrak{q}})$ is the decomposition group at $\mathfrak{q}$. If $\Sha^2_S(X) =  0$, then all the obstructions can be detected locally.

If the obstruction is trivial, and $\rho_{n+1}$ is a deformation of $\rho_n$, then $H^1(G_F, X)$ acts on the set of deformations. The action is given by $(f. \rho_{n+1}) (\sigma) = (I + p^n f(\sigma))(\rho_{n+1}(\sigma))$, where $f \in H^1(G_F, X)$. When $f$ is a coboundary, then $(I + p^n f(\sigma)) \rho_{n+1}(\sigma)$ is the same deformation. If the image of $\overline{\rho}$ is exactly the scalars, then $H^1(G, X)$ acts on the deformations of $\rho_n$ to $GL_2(W(\mathbf{k})/p^{n+1})$ as a principal homogeneous space. Deformations of residual Galois representations exist due to the work of Diamond, Taylor, Ramakrishna \cite{R1}, \cite{R3} and recently, Manoharmayum (\cite{Ma}) for number fields.

\section{The setup}

In this section we generalise some definitions and propositions from \cite{R4} and \cite{KLR} for a number field $F$, such that $[F(\mu_p):F] = p-1$. Since they are required for both sections, we sketch the proofs and refer the reader to the respective sources for more details.

\begin{definition}
Let $T$ be a finite set of primes and $X^*$ be the Cartier dual of $X$. If $L_{\mathfrak{q}} \subset H^1 (G_{\mathfrak{q}}, X)$ is a subspace with annihilator $L^{\bot}_{\mathfrak{q}} \subset H^1 (G_{\mathfrak{q}}, X^*)$ (under the local pairing), for $\mathfrak{q} \in T$, then we define $H^1_{\mathfrak{L} }(G_{F,T}, X)$ and $H^1_{\mathfrak{L^{\bot}} }(G_{F,T}, X^*)$ to be, respectively, the kernels of the restriction maps
\begin{itemize}
\item $H^1 (G_{F,T}, X) \rightarrow \oplus_{\mathfrak{q} \in T} H^1 (G_{\mathfrak{q}}, X) / L_{\mathfrak{q}}$
\item $H^1 (G_{F,T}, X^*) \rightarrow \oplus_{\mathfrak{q} \in T} H^1 (G_{\mathfrak{q}}, X^*) / L^{\bot}_{\mathfrak{q}}$
\end{itemize}
\end{definition}

These cohomology groups are also known as Selmer and dual Selmer groups, and the next proposition about them is due to Wiles \cite{W} (see also \cite{NSW}). We let $T$ be a finite of primes, which contains all primes of $F$ above rational primes dividing $\# X$.

\begin{proposition}
$dim H^1_{\mathfrak{L} }(G_{F,T}, X) - dim H^1_{\mathfrak{L}^{\bot} }(G_{F,T}, X^{*}) = dim H^{0} (G_{F,T}, X) - dim H^{0}(G_{F,T}, X) + \sum_{\mathfrak{q} \in T} (dim L_{\mathfrak{q}} - dim H^{0}(G_{\mathfrak{q}}, X))$
\end{proposition}

\begin{definition}

Let $q$ be a rational prime which splits completely in $F$ such that $q \not\equiv \pm 1 \mod p$, and $\mathfrak{q}$ be a prime lying above it in $F$. We say that $\mathfrak{q}$ is $\overline{\rho}$-nice if 

\begin{itemize}
\item $\overline{\rho}$ is unramified at $\mathfrak{q}$
\item the ratio of the eigenvalues of $\overline{\rho}(\sigma_{\mathfrak{q}})$ is $q$, where $\sigma_{\mathfrak{q}}$ is Frobenius at $\mathfrak{q}$.
\end{itemize}
\end{definition}

\begin{definition}
Let $R$ be a complete local Noetherian ring with residue field $\mathbf{k}$, and let $J$ be an ideal of finite index in $R$. If $\rho_{R/J}$ is a deformation of $\overline{\rho}$ to $GL_2(R/J)$, then $\mathfrak{q}$ is $\rho_{R/J}$-nice if 
\begin{itemize}
\item $\mathfrak{q}$ is nice for $\overline{\rho}$
\item $\rho_{R/J}$ is unramified at $\mathfrak{q}$ and the ratio of the eigenvalues of $\rho_{R/J}(\sigma_{\mathfrak{q}})$ is $q$.
\item $\rho_{R/J}(\sigma_{\mathfrak{q}})$ has the same (prime to $p$) order as $\overline{\rho}(\sigma_{\mathfrak{q}})$.
\end{itemize}
\end{definition}

Nice primes have useful properties, which are exhibited in the following propositions. 

\begin{proposition}
If $\mathfrak{q}$ is a nice prime, then the dimensions of $H^i (G_{\mathfrak{q}}, X)$ are $1,2,1$ for $i = 0,1,2$, respectively.
\end{proposition}

\begin{proof}
If $\mathfrak{q}$ is a nice prime, we see that as $G_{\mathfrak{q}}$-modules $X \cong \mathbf{k} \oplus \mathbf{k}(1) \oplus \mathbf{k}(-1)$, and $X^* \cong \mathbf{k}(1) \oplus \mathbf{k} \oplus \mathbf{k}(2)$, where $\mathbf{k}(r)$ accounts for the twist by the $r$-th power of the cyclotomic character. We assumed that $F$ doesn't contain the $p$-th roots of unity and $q \not\equiv \pm 1 \mod p$, so $H^i (G_{\mathfrak{q}}, \mathbf{k}(r)) = 0$ for $r\ne 0,1$ by local duality. The rest follows by using the local Euler-Poincare characteristic.
\end{proof}

\begin{proposition}
a) Nice primes have positive density.

b) We can find a finite set of nice primes $Q$ such that $\Sha^{i}_{S \cup Q}(X) = 0$ for $i = 1,2$.

c) Assume that $\Sha^{2}_T(X) = 0$. For any nice prime $\beta \notin T$, the inflation map $H^1(G_{F,T}, X) \rightarrow H^1 (G_{F,T \cup \{\beta \}}, X)$ has one dimensional cokernel.

\end{proposition}

\begin{proof}
a) Let $F(\overline{\rho})$ be the fixed field of the kernel of $\overline{\rho}$, $F(X)$ be the field fixed by the action of $G_F$ on $X = Ad^0 \overline{\rho}$, $K$ be the composite of $F(X)$ and $F(\mu_p)$, and $D$ be the intersection of the two fields.  To show that nice primes have positive density, we need to use Cebotarev's theorem and find the right conjugacy class associated to nice primes in $Gal(K/F)$. Note that $Gal(F(X)/D) \times Gal (F (\mu_p)/D) \cong Gal(K/D) \subseteq Gal (K/F)$, so we will define $C$ to be the conjugacy class of $a \times b \in Gal(F(X)/D) \times Gal (F (\mu_p)/D)$, where $a$ will correspond to nice primes.

Claim: $[D:F]=1$ or $2$.

As $D/F$ is abelian, the commutator subgroup of $Gal(F(\mu_p)/F)$, which is simply $SL_2(\bf{k})$, is contained in $Gal(F(\overline{\rho})/D)$. Since $D \subseteq F(X)$, the fixed field associated to $\overline{\rho}$, we see that $Gal(F(\overline{\rho})/D)$ contains the scalar matrices $Z$ in $Im \overline{\rho}$, so $Z.SL_2(\bf{k})$ is a normal subgroup of $Gal(F(\overline{\rho})/D)$. As $Im \overline{\rho}/SL_2(\bf{k})$ is cyclic, we see that $[Im \overline{\rho}: Z. SL_2 (\bf{k})]=$ $1$ or $2$. By definition $Gal(F(\overline{\rho})/F) = Im \overline{\rho}$, and $Z.SL_2(\mathbf{k}) \subseteq Gal(F(\overline{\rho})/D)$  so $[D:F]=1$ or $2$, which proves the claim.

Assume $p \geq 7$, and then choose $x \in \mathbb{F}_p^*$ such at $x^2 \ne \pm 1$. Consider $x^2$ in $Gal (F(\mu_p) / F)$. As $[D:F] = 1$ or $2$, we can see that $x^2 \in Gal (F(\mu_p) / D)$. $x^2$ will correspond to $b$ in our desired conjugacy class $C$. Let $\left(
                  \begin{array}{cc}
                    x & 0 \\
                    0 & x^{-1} \\
                  \end{array}
                \right) \in Gal(F(\overline{\rho})/D)$ and consider its projection $a$ in $Gal(F(X)/D)$. Then, $a$ corresponds to an element in the image of $\overline{\rho}$ whose eigenvalues have ratio $t \ne \pm 1 \in \mathbb{F}_p^*$

 Define $\alpha = a \times b \in Gal(F(X)/D) \times Gal (F (\mu_p)/D) \cong Gal(K/D) \subset Gal(K/F)$ and let $C$ be the conjugacy class of $\alpha$. By Cebotarev, we see that nice primes have positive density.

b) For the second part, we need to find a set of nice primes $Q = \{ \mathfrak{q}_1,.., \mathfrak{q}_m \}$ such that $f_i |_{G_\mathfrak{q_i}} \ne 0$ and $
g_j |_{G_\mathfrak{q_j}} \ne 0$ and $0$ otherwise, for any set of linearly indepedent $\{f_1,...,f_r\} \in \Sha^{1}_S(X)$ and $\{g_1,..,g_s \} \in \Sha^{1}_S(X^*)$. Let $K_{f_i}$ and $K_{g_j}$  be the fixed fields of the kernels of the restrictions of $f_i, g_j$ to $Gal(\overline{K}/K)$. Both these fields are linearly disjoint over $K$, so to find a nice prime $\mathfrak{q}$ which satisfies the required properties we need to make sure that the primes lying above it do not split completely from $K$ to $K_{f_i}$, but split completely from $K$ to $K_{f_j}$ for $i \ne j$ (similarly for $K_{g_j}$).
Such primes can be found because of the linear disjointedness over $K$ of $K_{f_i}$ and $K_{g_j}$. Now we choose a basis $<f_i>_{i \in I}$ of $\Sha^{1}_S(X)$ and $<g_j>_{j \in J}$ of $\Sha^{1}_S(X^*)$ and sets of primes $\{\mathfrak{q_i}\}_{i \in I}$ and $\{\mathfrak{p_j}\}_{j \in J}$ as above, so that $f_i (\sigma_{\mathfrak{q_i}}) \ne 0$ and $0$ otherwise (similarly for the $g_j$'s). Let $T = \{\mathfrak{q_i}, \mathfrak{p_j} \}_{i \in I, j \in J}$, and we see that $\Sha^{1}_{S \cup T}(X) = 0$ and $\Sha^{1}_{S \cup T}(X^*) = 0$. By duality, $\Sha^2_{S \cup T}(X) = 0$.

c) Let $L_{\mathfrak{q}} = H^{1}(G_{\mathfrak{q}}, X)$ for all $\mathfrak{q} \in T \cup \beta$. Using Prop 3.1, we see that the LHS is simply $H^{1}(G_F,T)$, and the RHS changes by $dim H^{1}(G_{\mathfrak{q}}, X) - dim H^{0}(G_{\mathfrak{q}}, X) = dim H^{2}(G_{\mathfrak{q}}, X)$. Using Prop 3.2, we see that for a nice prime $\mathfrak{q}$, $dim H^{2}(G_{\mathfrak{q}}, X) = 1$. 
\end{proof}

Now we define some deformation conditions.

\begin{definition}
Let $N_{\mathfrak{q}} = H^1 (G_{\mathfrak{q}}, \mathbf{k} (1)) \subset H^1 (G_{\mathfrak{q}}, X)$, and its annihilator $N^{\bot}_{\mathfrak{q}} = H^1 (G_{\mathfrak{q}}, \mathbf{k}(1)) \subset H^1 (G_{\mathfrak{q}}, X^*)$. We let $C_{\mathfrak{q}}$ be the class of deformations of $\overline{\rho}$ which are preserved by the action of $N_{\mathfrak{q}}$.
\end{definition}

The class of deformations $C_{\mathfrak{q}}$ correspond to the behaviour of the $\rho$'s when restricted to nice primes. This class of deformations behaves nicely and can be lifted to the next level easily. $N_{\mathfrak{q}}$ is called the set of null cohomology classes. When we act on the deformations by cohomology classes, the set $N_{\mathfrak{q}}$ is the set of classes which when restricted to nice primes leave the deformation untouched, as there are no obstructions to lifting at those primes. If there are obstructions, we will choose a cohomology class which is not in $N_{\mathfrak{q}}$ to overcome the obstructions. See \cite{R3} for more details on both $C_{\mathfrak{q}}$ and $N_{\mathfrak{q}}$.

Assume that $T$ is large enough so that $\Sha_T^2 (X) = 0$. Consider the restriction maps $\Psi_T : H^1 (G_{F,T}, X) \rightarrow \oplus_{\mathfrak{q} \in T} H^1 (G_{\mathfrak{q}}, X)$ and
$\Psi^*_T : H^1 (G_{F,T}, X^*) \rightarrow \oplus_{\mathfrak{q} \in T} H^1 (G_{\mathfrak{q}}, X^*)$. By global duality, we know that the images of $\Psi$ and $\Psi^*$ are exact annihilators of each other under the local pairing. This is also called the annihilation property.

\begin{definition}
(Local condition property) Assume that $T$ is large enough so that $\Sha^2_T (X) = 0$. Let $(z_{\mathfrak{q}})_{\mathfrak{q} \in T} \in \oplus_{\mathfrak{q} \in T} H^1 (G_{\mathfrak{q}}, X)$, such that $(z_{\mathfrak{q}})_{\mathfrak{q} \in T} \notin Im \psi_T$. For $\beta$ a nice prime, $h^{\beta} \in H^1 (G_{F, T \cup \{ \beta \}}, X)$ is a solution to the local condition problem if $h^{\beta} |_{G_{\mathfrak{q}}} = z_{\mathfrak{q}}$ for all $\mathfrak{q} \in T$.
\end{definition}

Remark: It is not hard to see that if $(z_{\mathfrak{q}})_{\mathfrak{q}} \notin Im  \psi_T$, then $\exists$ $\zeta \in H^1 (G_{F,T}, X^*)$ such that $\psi_T^* (\zeta)$ does not annihilate $(z_{\mathfrak{q}})_{\mathfrak{q} \in T}$.

\begin{proposition}
 Let $\rho_{R/J}$ be given, and $(z_{\mathfrak{q}})_{\mathfrak{q} \in T}$, and $\zeta$ be as above. Choose a basis $<\zeta_1, ..., \zeta_s>$ of $\psi_T^{* -1} (Ann (z_{\mathfrak{q}})_{\mathfrak{q} \in T})$. Let $Q$ be the set of nice primes such that 

\begin{itemize}
 \item $\zeta_i |_{G_{\mathfrak{q}}} = 0$.
\item $\zeta |_{G_{\mathfrak{q}}} \ne 0$
\item $f \in H^1 (G_{F,T}, X) \Rightarrow f|_{G_{\mathfrak{q}}} = 0$
\end{itemize}

Then, for any $\beta \in Q$ $\exists$ $h^{\beta} \in H^1 (G_{F,T \cup \{\beta\}}, X)$ which solves the local condition property.

\end{proposition}

\begin{proof}
 
Following Prop 3.4 of \cite{R4}, we see that $H^1_{\mathfrak{L}} (G_{F,T}, X)= H^1_{\mathfrak{L}}(G_{F,T \cup \{\beta\}}, X)$, where $\mathfrak{L_{\mathfrak{q}}} = 0$ for $\mathfrak{q} \in T$ and $\mathfrak{L_{\beta}} = H^1 (G_{\beta}, X)$ for $\beta \in Q$. 

From Prop 3.3, we see that the inflation map from $H^1 (G_{F,T}, X) \rightarrow H^1 (G_{F,T \cup \{\beta\}}, X)$ has a cokernel of dimension one, so $\exists$ $g \in H^1 (G_{F,T \cup \{\beta\}}, X)$ such that $(g |_{G_{\mathfrak{q}}})_{\mathfrak{q} \in T} \notin Im \psi_T$. We will modify this $g$ to get our required $h^{\beta}$.

Using the global reciprocity law, we have that $\sum_{\mathfrak{q} \in {T \cup \{ \beta \}}} inv_{\mathfrak{q}} (\zeta_i \cup g) = 0$ for any $i$. As $\zeta_i |_{G_{\beta}} = 0$, so $\sum_{\mathfrak{q} \in T} inv_{\mathfrak{q}} (\zeta_i \cup g) = 0$. We also chose $\zeta_i$ to be the annihilators of $z_{\mathfrak{q}}$, so $\sum_{\mathfrak{q} \in T} inv_{\mathfrak{q}} (\zeta_i \cup z_{\mathfrak{q}}) = 0$. Thus $\sum_{\mathfrak{q} \in T} inv_{\mathfrak{q}} (\zeta_i \cup (g - cz_{\mathfrak{q}})) = 0$, for any scalar $c$.

By the remark preceding this proposition, we know that $\sum_{\mathfrak{q} \in T} inv_{\mathfrak{q}} (\zeta \cup z_{\mathfrak{q}}) = b \ne 0$. Now we need to show that $\sum_{\mathfrak{q} \in T} inv_{\mathfrak{q}} (\zeta \cup g) = a \ne 0$. If it was zero, the local annihilation property implies that $(g |_{G_{\mathfrak{q}}})_{\mathfrak{q} \in T} \in Im \psi_T$, which contradicts the choice of $g$. So, $\sum_{\mathfrak{q} \in T} inv_{\mathfrak{q}} (\zeta \cup (g - \frac{a}{b} z_{\mathfrak{q}})) = 0$, which means that $\zeta$ annihilates $(g - \frac{a}{b}z_{\mathfrak{q}})_{\mathfrak{q} \in T}$.

If we choose $c = a/b$, then every element of $<\zeta_1, .., \zeta_s, \zeta>$ (which is a basis of $H^1 (G_{F,T}, X^*)$) annihilates $(g - \frac{a}{b} z_{\mathfrak{q}})_{\mathfrak{q} \in T} \Rightarrow \exists k \in H^1 (G_{F,T}, X)$ such that $\psi_T (k) = (g - \frac{a}{b}z_{\mathfrak
{q}})_{\mathfrak{q} \in T}$. Finally, set $h^{\beta} = \frac{b}{a}(g - k)$ which solves the local condition property.

\end{proof}

\begin{proposition}
Let $\rho_{R/J}, Q, h^{\beta}$ be as in the previous proposition. Then there exists a $\beta \in Q$ such that $h^{\beta} |_{G_{\beta}} \in N_{\beta}$ or there exist $\beta_1, \beta_2 \in Q$ and $h = \alpha_1 h^{\beta_1} + \alpha_2 h^{\beta_2}$ such that $h |_{\beta_i}  \in N_{\beta_i}$ for $i=1,2$ and $h$ solves the local condition property.
\end{proposition}

\begin{proof}

From Prop 3.2, we know that $H^1 (G_{\beta}, X) \cong H^1 (G_{\beta}, \mathbf{k}) \oplus H^1 (G_{\beta}, \mathbf{k}(1)) $. Recall that we defined $N_{\beta} = H^1 (G_{\beta}, \mathbf{k}(1))$, so if $h^{\beta} \in N_{\beta}$ for some $\beta \in Q$, we're done. If not, we need to find $\beta_1, \beta_2$ as stated. Consider the matrix consisting of the values of $h^{\sigma_{\beta_i}} (\beta_j)$ for $i,j = 1,2$:

\begin{displaymath}
% use packages: array,booktabs
\left( \begin{array}{lll}
 & \sigma_{\beta_1} & \sigma_{\beta_2} \\ 
h^{\beta_1} & a & b \\ 
h^{\beta_2} & c & d
\end{array}\right)
\end{displaymath}

Since we are assuming that one prime $\beta$ doesn't exist as required, $a,d \notin N_{\beta_i}$ (we say that $a,d \ne 0$ for this condition). To prove the theorem we need to show that we can find a linear combination $h = \alpha_1 h^{\beta_1} + \alpha_2 h^{\beta_2}$ such that $h |_{\beta_i}  \in N_{\beta_i}$ for $i=1,2$ and $h$ solves the local condition property. As $h^{\beta_1} , h^{\beta_2} $ solve the local condition property, to show that $h$ solves the local condition property we need $\alpha_1 + \alpha_2 = 1$. Thus, we need to prove that the above matrix has unequal rows and zero determinant.

Let $y$ be the value of $h^{\beta} (\sigma_{\beta})$ that occurs most often. Then the set $Y = \{ \beta \in Q | h^{\beta} (\sigma_{\beta}) = y \}$ has positive upper density ($Q$ as noted earlier is a Cebotarev set). For a nice prime $\beta$ choose $\eta_{\beta} \in H^1 (G_{\beta}, X^*)$ such that $ inv_{\beta} (\eta_{\beta} \cup h^{\beta}) \ne 0$. Let $z$ be the value that occurs most often. Then $Z = \{\beta \in Y | inv_{\beta} (\eta_{\beta} \cup h^{\beta})  = z \}$ also has positive upper density. 

Choose any $\beta_1 \in Z$. Then, we need to find a $\beta_2 \in Z$ such that $ad - bc = 0$, and $b,c \ne y$. To choose $b (=h^{\beta_1} (\sigma_{\beta_2}))$ to be whatever we want is basically a Cebotarev condition on $\beta_2$, corresponding to $h^{\beta_1}$. Now we need to choose $c (=h^{\beta_2} (\sigma_{\beta_1}))$ appropriately. By Prop 3.6 of \cite{R4}, we see that $\frac{h^{\beta_2} (\sigma_{\beta_1})}{y}. \frac{1}{p} = \frac{1}{p} - \xi^{\beta_1}(\sigma_{\beta_2})z$, where $\xi^{\beta_1} \in H^{1}_{P^*} (G_{F,T \cup \{\beta\}}, X^*)$, $P_{\mathfrak{q}} = 0$ for $\mathfrak{q} \in T$ and $P_{\mathfrak{q}} = N_{\mathfrak{q}}$. Thus, we need to choose $\xi^{\beta_1}(\sigma_{\beta_2})$ appropriately. When we consider $h^{\beta_1}$ and $\xi^{\beta_1}$, we see that due to the linear disjointedness of the field extensions associated to them (as in prop 3.3), they both give independent Cebotarev conditions.

Suppose there is no $\beta_2 \in Z$ for which $h^{\beta_1} (\sigma_{\beta_2})$ and $\xi^{\beta_1}(\sigma_{\beta_2})$ can be what we want. Then the set $Z  \setminus \{ \beta_1 \}$ is in a Cebotarev class which is complementary to the Cebotarev conditions imposed on $\sigma_{\beta_2}$ due to the choice of $h^{\beta_1} (\sigma_{\beta_2}) = x \ne 0,y$. Let $D$ be the density of $Q$. Then the two complementary Cebotarev classes have density $D\gamma$ where $\gamma < 1$.

If we replace $\beta_1$ by a sequence of primes $\wp_i \in Z$ which don't allow us to choose the appropriate second prime, then $Z \setminus \{\wp_i \} $ lies in a Cebotarev class complementary to the Cebotarev class associated to $h^{\wp_i}$ and $\xi^{\wp_i}$, which are all independent of each other. By imposing $n$ such conditions, we see that the density of these complementary Cebotarev classes is $D \gamma^n$. As $n$ gets larger and larger, we see that the set $Z \setminus \{\wp_1, ..., \wp_n \}$ lies inside a set of density zero, which is a contradiction. Thus we can always find 2 primes $\beta_1, \beta_2$ which satisfy the hypothesis of the proposition.

\end{proof}

\begin{proposition}
 Let $B$ be a set of one or two nice primes such that $H^1_{N} (G_{S_n}, X) \hookrightarrow H^1_{N} (G_{F,S_n \cup B}, X)$ has cokernel of dimension $1$. Let $A$ be the set of one or two nice primes from the previous proposition and let $(z_{\mathfrak{q}})_{\mathfrak{q} \in S_n \cup B} \notin \oplus_{\mathfrak{q} \in S_n \cup B} N_{\mathfrak{q}} \oplus \psi_{S_n \cup B} (H^1 (G_{F,S_n \cup B}, X))$. 

Then, $H^1_{N} (G_{F,S_n \cup B \cup A}, X) = H^1_{N} (G_{F,S_n \cup B}, X)$.

\end{proposition}

\begin{proof}
 
In the hypothesis of Prop 3.4, $f \in H^1 (G_{F,T}, X) \Rightarrow f|_{G_{\mathfrak{q}}} = 0$. As $N_{\mathfrak{q}} = 0$, for $\mathfrak{q} \in S_n \cup B$ and $f \in N_{\mathfrak{q}}$ for $\mathfrak{q} \in A$, we see that $H^1_{N} (G_{F,S_n \cup B}, X) \subset H^1_{N} (G_{F,S_n \cup B \cup A}, X)$. For the other containment we consider the 2 cases for $A$:

1) $A$ has one prime $\beta$:

Any element of $H^1_{N} (G_{F,S_n \cup B \cup A}, X) \setminus H^1_{N} (G_{F,S_n \cup B}, X)$ looks like $f + \alpha h^{\beta}$, where $f \in H^1 (G_{F,S_n \cup B}, X)$, $\alpha \ne 0$ and $h^{\beta}$ solves the local condition property i.e. $h^{\beta} |_{G_{\mathfrak{q}}} = z_{\mathfrak{q}}$ for $\mathfrak{q} \in S_n \cup B$. Thus, $ (f + \alpha h^{\beta})|_{G_{\mathfrak{q}}} = f + \alpha z_{\mathfrak{q}} \in N_{\mathfrak{q}}$, for all $\mathfrak{q} \in S_{n} \cup B \cup A$. But we assumed that $z_{\mathfrak{q}} \notin N_{\mathfrak{q}} \oplus Im(\psi_{S_n \cup B}) \Rightarrow \alpha = 0$, a contradiction.

2) $A$ has 2 primes $\beta_1, \beta_2$:

Any element of $H^1_{N} (G_{F,S_n \cup B \cup A}, X) \setminus H^1_{N} (G_{F,S_n \cup B}, X)$ looks like $f + \alpha_1 h^{\beta_1} + \alpha_2 h^{\beta_2}$, where $f \in H^1 (G_{F,S_n \cup B}, X)$. Using a similar argument as in case 1, we see that $\alpha_1 + \alpha_2 = 0$. We know that $\alpha_1 (h^{\beta_1} - h^{\beta_2})|_{G_{\mathfrak{q}}}= 0$ for all $\mathfrak{q} \in S_n \cup B$. This means that $f|_{G_{\mathfrak{q}}} \in N_{\mathfrak{q}}$ for all $\mathfrak{q} \in S_n \cup B$ $\Rightarrow f \in H^1_N (G_{F,S_n \cup B}, X)$. We also know that $f|_{G_{\beta_{i}}} = 0$ for $i = 1,2$, so $f \in H^1_N (G_{F,S_n \cup B \cup A}, X)$. Thus, $\alpha_1 (h^{\beta_1} - h^{\beta_2}) \in H^1_N (G_{F,S_n \cup B \cup A}, X) \Rightarrow \alpha_1 ( h^{\beta_1} - h^{\beta_2} ) |_{G_{\beta_i}} \in N_{\beta_i}$, for $i = 1,2$. By the construction of $\beta_i$ in the previous proposition, we see that $\alpha_1 = 0 \Rightarrow f \in H^1_N (G_{F,S_n \cup B}, X)$, which is a contradiction.

\end{proof}

\begin{proposition}
 There exists a set $B$ of one or two $\rho_n$-nice primes such that the map $H^1_N(G_{F, S_n},X) \hookrightarrow H^1_N(G_{F, S_n \cup B},X)$ has one dimensional cokernel. 
\end{proposition}

\begin{proof}
By a straight generalisation of prop 10 of \cite{KLR} we can find a set $B = \{\beta_1, \beta_2 \}$ of one or two $\rho_n$-nice primes for which there is a linear combination $f = \alpha_1 f_{\beta_1} + \alpha_2 f_{\beta_2}$ such that $f(\sigma_{\beta_1}) = 0$ for $i =1,2$ and $f|_{G_{\mathfrak{q}}} = 0$ for $\mathfrak{q} \in S_n$. As $\beta_1, \beta_2$ are nice primes, we see that $N_{\beta_{i}} = 0$ for $i =1,2$, which means that $f$ is in the kernel of the map $H^1 (G_{F, S_n \cup B},X) \rightarrow \oplus_{\mathfrak{q} \in S_n \cup B} H^1(G_{\mathfrak{q}},X) / N_{\mathfrak{q}}$. 
\end{proof}

\section{Large image}

We want to construct $\rho: G_F \rightarrow GL_2 (W(\mathbf{k})[[ T_1,...,T_r,...]])$ such that $Im \rho \supseteq SL_2 (W(\mathbf{k})[[ T_1,...,T_r,...]])$. We start with $\overline{\rho} : G_F \rightarrow GL_2 (\mathbf{k})$ and lift it successively to $\rho_n : G_{F, S_n} \rightarrow GL_2 (W(\mathbf{k})[[ T_1,...,T_n]] / (p, T_1,...,T_n)^n)$, and define $\rho = \underleftarrow{lim }\rho_n$ such that at each stage $n$, $Im \rho_n \supseteq SL_2 (W(\mathbf{k})[[ T_1,...,T_n]]/(p, T_1,...,T_n)^n)$.

If $R_n$ is the deformation ring of $\rho_n$, with $m_{R_n}$ the maximal ideal, then we see that $R_n / m^n_{R_n} = (W(\mathbf{k})[[ T_1,...,T_n]] / (p, T_1,...,T_n)^n)$. We will add more primes of ramification to $S_n$ and get a new set of primes $S_{n+1}$, such that the deformation ring associated to $S_{n+1}$ has $R_{n+1}/m^{n+1}_{R_{n+1}}$ as a quotient. This gives us a surjection from $R_{n+1} / m_{R_{n+1}}^{n+1} \twoheadrightarrow R_n /m_{R_n}^n$, which allows us to get the inverse limit $\rho = \underleftarrow{lim} \rho_n$.

We assume that there exists $\rho_n : G_{F, S_n} \rightarrow GL_2 (R_n / m^n_{R_n})$, and $dim H^1_N (G_{F, S_n}, X) = n$. By Prop 3.7 we can find a set $B$ of $\rho_n$-nice primes such that $dim H^1_N (G_{F, S_n \cup B}, X) = n +1$. Let $U$ be the deformation ring associated to the augmented set $S_n \cup B$, with the deformation conditions ($N_{\mathfrak{q}}, C_{\mathfrak{q}}$). As $B$ consists of $\rho_n$-nice primes, we have a surjection $ \phi : U \twoheadrightarrow  R_n /m_{R_n}^n$ which means that for some $I_1$, we $U / I_1 = R_n / m_{R_n}^n$. As $dim H^1_N (G_{F, S_n \cup B}, X) = n +1$, we see that as a ring $U$ consists of power series of $(n+1)$ variables. Thus, for some $I_2$, $U/I_2 = \mathbf{k}[[T_1,...,T_{n+1}]]/(T_1,...,T_{n+1})^2 $. Let $I = I_1 \cap I_2$, and define $U_0 = U /I$. 

Our goal is to get a deformation ring which has $R_{n+1}/m_{R_{n+1}}^{n+1}$ as a quotient. If $U_0$ is such a deformation ring, we're done. If not, we get a sequence 

$R_{n+1}/m_{R_{n+1}}^{n+1} \twoheadrightarrow ... \twoheadrightarrow U_1 \twoheadrightarrow U_0$, 

where the kernel at each stage has order $p$. Then we add more primes of ramification to $S_n \cup B$ so that the augmented deformation ring has $U_1$ as a quotient. We then keep iterating to get our required deformation ring.

As $U_0$ is a quotient of $U$, we let $\rho_{U_0}$ be the induced deformation. Since $B$ consists of nice primes and all global obstructions can be detected locally ($\Sha^2_{S_n}(X) = 0$ even as $S_n$ gets bigger) we see that $\rho_{U_0}$ can be lifted to $U_1$, and we call the deformation $\tilde{\rho_{U_1}}$. If there are no local obstructions to lifting $\tilde{\rho_{U_1}}$, we're done. If there are local obstructions, then we choose a set of cohomology classes $(z_{\mathfrak{q}})_{\mathfrak{q} \in S_n \cup B}$ such that the action of $z_{\mathfrak{q}}$ on $\tilde{\rho_{U_1}}|_{z_{\mathfrak{q}}}$ overcomes the local obstructions at $\mathfrak{q} \in S_n \cup B$. By Prop 3.5 we can find a set $A$ and a cohomology class $h$ such that:

\begin{itemize}
 \item $\mathfrak{q}$ is $\tilde{\rho_{U_1}}$-nice, for $\mathfrak{q} \in A$ (no new obstructions at $A$)
 \item $h|_{G_{\mathfrak{q}}} \in N_{\mathfrak{q}}$ for $\mathfrak{q} \in A$ ($h$ preserves the class of deformations)
 \item $h|_{G_{\mathfrak{q}}} = z_{\mathfrak{q}}$, for $\mathfrak{q} \in S_n \cup B$ ($h$ overcomes local obstructions at $S_n \cup B$)
\end{itemize}

We now define $\rho_{U_1} = (I + h)\tilde{\rho_{U_1}}$. There are no obstructions to lifting $\rho_{U_1}$ and the augmented deformation ring has $U_1$ as a quotient. By Prop 3.6 we know that $dim H^1_N (G_{F, S_n \cup B}, X) =  dim H^1_N (G_{F, S_n \cup B \cup A}, X) = n+1 $, so we can keep iterating in a similar way by adding more primes, and get a set $S_{n+1}$ which has $R_{n+1}/m_{R_{n+1}}^{n+1}$ as a quotient.

It remains to show that at each stage $n$, $Im \rho_n \supseteq SL_2(R_n / m_{R_n}^n)$. We recall a proposition due to Boston \cite{Bo}.

\begin{proposition}
 Let $R$ be a complete noetherian local ring with maximal ideal $m$, with $R/m$ finite, char$R/m = p \ne 2$. Let $H$ be a closed subgroup  of $SL_n (R)$ projecting onto $SL_n (R/m^2)$. Then $H = SL_n (R)$.
\end{proposition}

For our purpose, we let $H = Im \rho_n \cap SL_2 (R_n/m_{R_n}^n)$. To show that $Im \rho_n \supseteq SL_2 (R_n / m_{R_n}^n)$, all we need to show is that $Im \rho_n / m_{R_n}^2 \supseteq SL_2 (R_n / m_{R_n}^2)$.

Consider the exact sequence $SL_2 (m_{R_m}/ m_{R_m}^2) \xrightarrow{\theta} SL_2 (R_n / m_{R_n}^2) \xrightarrow{\phi} SL_2 (R_n / m_{R_n})$. As a $\mathbf{k} [Im \overline{\rho}]$-module, $Ker \phi$ has $(n+1)$ copies of $X$. To see this, note that Ker$\{SL_2 (W(\mathbf{k})/p^2) \rightarrow SL_2(\bf{k}) \}$ consists of one copy of $X$ ($R_n / m_{R_n} = \mathbf{k}$). $R_n / m_{R_n}^2$ consists of power series like $ a_0 + a_1 T_1 + ...+ a_n T^n$, where $a_0 \in W(\mathbf{k})/p^2$ corresponds to one copy of $X$, and the remaining $n$ $a_i$'s ($\in \bf{k}$) account for the other $n$ copies. 

Now we look at the projection $ Im \rho_n / m_{R_n}^2 \rightarrow Im \overline{\rho} $. Recall that we assumed $Im \overline{\rho} \supseteq SL_2 (\mathbf{k})$, so we need to obtain $(n+1)$ copies of $X$ in the kernel to show that $Im \rho_n / m_{R_n}^2$ is big enough. We know there is an equivalence between $Ext^1 (X,X)$ (killed by $p$), $H^1 (G_F, X)$ and lifts of $\overline{\rho}$ to the dual numbers $\bf{k}[\epsilon] $ . Since $\rho_n$ is a deformation of $\overline{\rho}$ and we assumed that $dim H^1_N (G_F, X) = n$, we get $n$ copies of $X$. These give split extensions as we are considering lifts to the dual numbers and $\bf{k}$ embeds
    in $\bf{k}[\epsilon]$.  To get the last copy of $X$, we see that for $n \geq 2$, the deformation of $\overline{\rho}$ to $GL_2 (W(\mathbf{k}) / p^2)$ gives us an extension of $X$. The exact sequence $SL_2 (p W(\mathbf{k}) / p^2) \hookrightarrow SL_2 (W(\mathbf{k}) / p^2) \twoheadrightarrow SL_2 (W(\mathbf{k}) / p)$ doesn't split, so this corresponds to a non-split extension and see that $Im \rho_n / m^2_{R_n} \supseteq SL_2 ( R_n / m_{R_n}^2)$.

Now we are in a position to state the final theorem.

\begin{theorem}
There exists a deformation of $\overline{\rho}$, ramified at infinitely many primes, $\rho : G_F \rightarrow GL_2 (W(\mathbf{k})[[T_1, T_2,.., T_r,....,]])$, whose image contains $SL_2 (W(\mathbf{k})[[T_1, T_2,.., T_r, ....,]])$.
\end{theorem}

\begin{proof}
All that we need to show is that $\rho = \underleftarrow{lim} \rho_n$ is full  i.e. given $\alpha \in SL_2 (R)$, $\exists$ $g \in G_F$ such that $\rho (g) = \alpha$ , where $R = \underleftarrow{lim} R_n / m_{R_n}^n $. 

At each stage $n$ we know that $Im \rho_n \supseteq SL_2 (R_n/m_{R_n}^n)$. So if $\alpha_n \in R_n / m_{R_n}^n$ is the image of $\alpha$, then there exists a $g_n \in G_F$ such that $\rho_n (g_n) = \alpha_n$. Ideally we would like to define $g = \displaystyle\lim_{\overrightarrow{n}} g_n$, and take limits again but we need to modify $g_n$ a little bit. As $G_F$ is compact and Hausdorff, there exists a subsequence $g_{n_m}$ of $g_n$, such that $g_{n_m}$ has $g$ as a limit point. 

For $n_m > n$, we know that $\rho_{n_m}$ is a deformation of $\rho_n$. So when we reduce $\rho_{n_m}$ mod $m_{R_n}^n$, we see that $\rho_n (g_{n_m}) =\alpha_n$. Also, as $\rho_n$ is continuous we can switch limits and see that $\displaystyle\lim_{\overleftarrow{n}} \rho_n (\displaystyle\lim_{\overrightarrow{m}}g_{n_m}) = \displaystyle\lim_{\overleftarrow{n}}   \displaystyle\lim_{\overrightarrow{m}}\rho_n(g_{n_m})$.

Finally, $\rho(g) = \displaystyle\lim_{\overleftarrow{n}} \rho_n(g) = \displaystyle\lim_{\overleftarrow{n}} \rho_n(\displaystyle\lim_{\overrightarrow{m}} g_{n_m}) = \displaystyle\lim_{\overleftarrow{n}}\displaystyle\lim_{\overrightarrow{m}}\rho_n (g_{n_m}) = \displaystyle\lim_{\overleftarrow{n}} (\displaystyle\lim_{\overrightarrow{m}} \alpha_{n}) = \displaystyle\lim_{\overleftarrow{}n} \alpha_n = \alpha$.

\end{proof}

\section{Distribution Statements}

We start with a residual representation $\overline{\rho}: G \rightarrow GL_2(\mathbb{Z}/ p)$, and lift it successively to $\rho_n : G \rightarrow GL_2 (\mathbb{Z}/p^n)$. At each stage we will add more primes $Q_n$ to the ramification set $S_n$, and fix the characteristic polynomials for a finite, but large, set of unramified primes $R_n$. To overcome local obstructions to lifting at the primes that we keep adding we will need to add some more primes $A_n$. This will allow the lifts to be unobstructed at the ramified primes $S_n$ and keep the characteristic polynomials unchanged at the unramified primes $R_n$. As we keep lifting $\overline{\rho}$ to $\rho_n$, we will control the cardinality at each stage for $R_n$ so that after taking limits the cardinality is large enough to get a representation $\rho : G \rightarrow GL_2(\mathbb{Z}_p)$ for which the distribution of the $a_l$'s is in contrast to the distribution statements in the first section.

Let $\overline{\rho}: G \rightarrow GL_2(\mathbb{Z}/p)$ be a surjective Galois representation, with $p \geq 5$ that arises from $S_2(\Gamma_0(N))$ for some $(N,p) = 1$ and $N$ square-free. By results of Serre, we know that such representations exist. If we take a semistable elliptic curve $E$ over $\mathbb{Q}$, then for all but finitely many primes the corresponding mod $p$ representation satisfies the required conditions. Let $S$ be a finite set of primes containing $p$ and the primes of ramification of $\overline{\rho}$. The key lemma that allows us do this is Lemma 8 from \cite{KLR}:

\begin{lemma}
Let $\rho_n$ be a deformation of $\overline{\rho}$ to $GL_2(\mathbb{Z}/p^n)$, unramified outside a set $S$, and assume $\Sha^1_{S}(X)$ and  $\Sha^2_{S}(X)$ are trivial. Let $R$ be any finite collection of unramified primes of $
\overline{\rho}$ disjoint from $S$. Then, there is a finite set $Q = \{q_1, q_2,..., q_m \}$ of $\rho_n$-nice primes disjoint from $R \cup S$, such that we get the following isomorphisms:

\begin{itemize}

\item $H^1(G_{S \cup R \cup Q}, X^*) \rightarrow \oplus_{v \in Q} H^1(G_v, X^*) $

\item $H^1(G_{S \cup Q}, X) \rightarrow (\oplus_{s \in S} H^1(G_s, X)) \oplus (\oplus_{r \in R} H^1_{nr} (G_r,X))$

\end{itemize}

($H^1_{nr}(G_v, X) = Im {H^1(G_v/I_v, X^{I_v}) \rightarrow H^1(G_v, X)}$, where $I_v \subset G_v$ is the inertia group)

\end{lemma}

The above lemma allows us to find a global cohomology class such that its action on a deformation overcomes obstructions at primes of $S$ and allows us to choose the characteristic polynomials at $R$ as we want. Now we might have added local obstructions at the primes of $Q$. To overcome these obstructions we use Prop 3.5 to add a set $A$ of one or two nice primes and a global cohomology class $h$ which overcomes local obstructions at primes of $Q$, leaves the characteristic polynomials at $R$ unchanged, and doesn't add any new obstructions at $A$.

We now state the main proposition (based on \cite{KLR}), which we use to build the required representations. The subtle difference here is the use of the local condition property to find the set of primes $A$ and one global cohomology class $h$, rather than working individually with each local condition.

\begin{proposition}

There exists a deformation $\rho = \underleftarrow{lim} \rho_n$ of $\overline{\rho}$, which is ramified at infinitely many primes such that the characteristic polynomials of each $\rho_n$ can be chosen arbitrarily at the set of ramified primes.

\end{proposition}

\begin{proof}

We assume that $S$ is large enough so that $\Sha^1_S(X)$ and $\Sha^2_S(X)$ are trivial.We  We proceed to work inductively, and build $\rho$ from $\rho_n : G_{S_n} \rightarrow GL_2 (\mathbb{Z}/p^n)$.

$\mathbf{n=2:}$

Let $S_2 = S$. As $\Sha_S^2(X)= 0$, we see that there are no global obstructions to lifting $\overline{\rho}$ to $GL_2(\mathbb{Z}/p^2)$, so we let $\rho_2 : G_S \rightarrow GL_2(\mathbb{Z}/p^2)$ be any deformation of $\overline{\rho}$. We choose any deformation of $\overline{\rho} $ to $GL_2(\mathbb{Z}_p)$, and call it $\rho^{S_2}$. Let $R_2$ be a finite set of primes for which we choose the characteristic polynomials consistent with those of $\overline{\rho}$. By Lemma 5.1, we can find a set $Q_2$ for which there is a unique $f_2 \in H^1(G_{S_2 \cup Q_2}, X)$, such that $(I + p f_2)\rho_2|_{G_{v}} \equiv \rho^{S_2}|_{G_{v}} \mod p^2$. As $f_2|_{G_{r}} \in H^1_{nr}(G_r, X)$ for $r \in R_2$, we can see that the characteristic polynomials of this new representation remain the same for the primes in $R_2$. By adding the set $Q_2$ we may have some new obstructions. To overcome the obstructions we use Prop 3.5 to add a set $A_2$ to the ramification set and choose $h_2 \in H^1(G_{S_2 \cup Q_2 \cup R_2 \cup A_2})$ such that $(I + p(f_2 + h_2))\rho_2$ is unobstructed at the new primes, and still has the same characteristic polynomials for $R_2$.

$\mathbf{n=3:}$

Since adding primes to the set of ramification doesn't add new global obstructions ($\Sha^2_{S_2 \cup Q_2} = 0$), we can lift $(I + p(f_2 + h_2))\rho_2$ to a representation $\rho_3 : G_{S_3} \rightarrow GL_2(\mathbb{Z}_p)$, where $S_3 = S_2 \cup Q_2 \cup A_2$. For the set $S_3 \backslash S_2$, choose a deformation $\rho^{S_3 \backslash S_2}$ to $GL_2 (\mathbb{Z}_p)$, and also fix characteristic polynomials for the set $R_3 \backslash R_2$. As in the case n=2, we can add two sets of primes $Q_3, A_3$ to the ramification set and get $f_3 \in H^1(G_{S_3 \cup Q_3}, X)$ and $h_3 \in H^1(G_{S_3 \cup Q_3 \cup R_3 \cup A_3}, X)$, such that $(I + p^2(f_3 + h_3))\rho_3 \equiv  \rho \mod p^3$, when restricted to the primes $v \in S_3$. Also, $(I + p^2(f_3 + h_3))\rho_3$ is unobstructed at the new primes we have added and has the same characteristic polynomials at primes in $R_3$.

$\mathbf{n=4}$

Since there are no obstructions to lifting, we lift $(I + p^2(f_3 + h_3))\rho_3$ to $\rho_4$, and continue the same argument. Hence, we get representations $\rho_n : G_{S_n} \rightarrow GL_2(\mathbb{Z}/p^n)$, which are unramified away from $S_n$, and have fixed characteristic polynomials for primes in $R_n$. Let $\rho = \underleftarrow{lim} \rho_n$, and $R = \underrightarrow{lim} R_n$. By a result of Khare-Rajan \cite{KR}, we know that the set $R$ of unramified primes is of density one.

\end{proof}

We now use this construction of $\rho$ to prove our results about the distributions of the $a_l$'s. All we have to do is choose the set of unramified primes $R_n$ carefully, as well as bounds for it (since it's a finite set of primes) to get our following results.

\begin{corollary}
For any $\epsilon \leq 1$, there exists a deformation $\rho : G \rightarrow GL_2 (\mathbb{Z}_p)$ ramified at infinitely many primes, such that the set $R = \{ l | \frac{a_l}{2\sqrt{l}} \leq \epsilon \}$ is of density one.
\end{corollary}

\begin{proof}
At each stage $n$, we define $R_n = \{ l, (\frac{c}{2}p^n)^2 < l < (\frac{c}{2}p^{n+1})^2 | a_l^2 - \frac{4}{c^2} l < 0  \} $, where $c$ is an arbitrary constant. Since each $a_l$ is determined $\mod p^n$, this is possible. The Sato-Tate formula tells us that $\frac{a_l}{2\sqrt{l}}$ lies in the interval $[-1,1]$, and they follow a particular distribution. By choosing $c > 1$, we are shrinking the distribution between $[\frac{-1}{c},\frac{1}{c}]$. Now, let $\epsilon = \frac{1}{c}$. As we choose our $c$ to be bigger (or $\epsilon$ to be smaller) the set of unramified primes $R_n$ also gets bigger but stays finite, which allows us to use the methods above to get the set $R = \underrightarrow{lim} R_n$.
\end{proof}

As is evident, the larger the value of $c$, the smaller the interval at which the $\alpha_l = \frac{a_l}{2\sqrt{l}}$ lie. Hence we get a probability distribution for the $\alpha_l$'s which is contrary to the Sato-Tate conjecture.

\begin{corollary}
If $D \in \mathbb{Z}$ is fixed, we can find a deformation $\rho : G \rightarrow GL_2 (\mathbb{Z}_p)$, such that the set $R(x) = \{ l < x | a_l = D \} \ne O(\frac{\sqrt{x}}{log x})$ \end{corollary}

\begin{proof}
We define $R_n = \{ l < x_n | a_l \equiv D \mod p^n\}$, where we choose $x_n$ to be large enough such that $\# R_n > \frac{x_n^{1-2/n}}{log (x_n)}$. As each $R_n$ is a Cebotarev set, we can always find a large enough $x_n$. At the $(n+1)$th stage, the $x_{n+1}$ will be substantially bigger than $x_n$, but these sets have positive density so we can keep choosing the $x_n$'s to be large enough. When we take limits $R = \{l  |  a_l = D \} = \underrightarrow{lim} R_n$, is a set of zero density, but $\# R(x) = O(\frac{x^{1- \epsilon}}{log x})$, where $0 < \epsilon < 1$.
\end{proof}

The above corollary gives us an asymptotic formula contrary to the Lang-Trotter conjecture. The following corollaries will give contrary asymptotic formulae for the distribution of supersingular and anomalous primes.

\begin{corollary}
There exists a deformation $\rho : G \rightarrow GL_2 (\mathbb{Z}_p)$, such that the set of supersingular primes, $R = \{l <x | a_l = 0 \} \ne O(x^{3/4})$.  \end{corollary}

\begin{proof}
We choose $D = 0$ and let $\epsilon < 1/4$ in the previous corollary, and see by straightforward computations that $\# R = \{ l < x | a_l = 0 \} > O ( x^{3/4})$.
\end{proof}

\begin{corollary}
There exists a deformation $\rho : G \rightarrow GL_2 (\mathbb{Z}_p)$, such that the set of anomalous primes, $R = \{l <x | a_l = 1 \} \ne (\frac{x}{(log x)^{1+\epsilon}})$, where $\epsilon < 1/3$.\end{corollary}

\begin{proof}
We choose $D = 1$, and define  $R_n = \{ l < x_n | a_l \equiv 1 \mod p^n\}$, where $x_n$ is large enough so that $\# R_n  > \frac{x_n}{log (x_n)^{1 + 1/2^n}}$. As we take limits, we see that $\# R > O (\frac{x}{(log x)^{1+\epsilon}})$
\end{proof}

\section{Concluding remarks}

It would be interesting to know if it is possible to manipulate the deformations so that we can get any distribution we want for the $a_l$'s. The main problem we encountered in using the above methods was trying to control infinitely many values for the $a_l$'s with only finitely many cocycles at each stage.

$\bf{Acknowledgements}$: The author would like to thank Ravi Ramakrishna for suggesting the problem and also for many useful conversations and comments about this paper. Finally, the Cornell mathematics department for the visiting position for the academic years of 2007 and 2008, during which the majority of the work on this paper was done.

Aftab Pande\\
Universidade Federal do Rio de Janeiro\\
Email: aftab.pande@gmail.com\\

\end{document}